\documentclass[11pt]{amsart}

\usepackage{amscd,amssymb,amsopn,amsmath,amsthm,graphics,amsfonts,enumerate,verbatim,calc}
\usepackage[all]{xy}
\usepackage{tikz-cd}
\usepackage{hyperref}

\hypersetup{
	colorlinks=true,
	linkcolor=blue,
	citecolor=red,
	urlcolor=cyan,
}

\numberwithin{equation}{section}
\newtheorem{theorem}{Theorem}[section]
\newtheorem{lemma}[theorem]{Lemma}
\newtheorem{proposition}[theorem]{Proposition}
\newtheorem{corollary}[theorem]{Corollary}
\newtheorem{question}[theorem]{Question}

\theoremstyle{definition}

\theoremstyle{remark}
\newtheorem{remark}[theorem]{Remark}

\newcommand{\Ass}{\operatorname{Ass}}

\newcommand{\Spec}{\operatorname{Spec}}
\newcommand{\Ht}{\operatorname{ht}}

\newcommand{\Ext}{\operatorname{Ext}}
\newcommand{\Supp}{\operatorname{Supp}}

\newcommand{\Att}{\operatorname{Att}}
\newcommand{\Ann}{\operatorname{Ann}}
\newcommand{\depth}{\operatorname{depth}}

\newcommand{\Var}{\operatorname{Var}}
\newcommand{\Psupp}{\operatorname{Psupp}}
\newcommand{\p}{\mathfrak p}
\newcommand{\q}{\mathfrak q}
\newcommand{\m}{\mathfrak m}
\newcommand{\n}{\mathfrak n}

\newcommand{\R}{\widehat R}

\begin{document}
	\title{The Depth Formula for modules over quotients of Gorenstein rings }
	
	\author{Tran Nguyen An}
	\address{Thai Nguyen University of Education, Thai Nguyen, Vietnam}
	\email{antn@tnue.edu.vn; antrannguyen@gmail.com}
	
	\author[Pham Hung Quy]{Pham Hung Quy}
	\address{Department of Mathematics, FPT University, Hanoi, Vietnam}
	\email{quyph@fe.edu.vn}
	
	\subjclass[2020]{Primary 13D45, 13C15, 13H10}
	\keywords{Depth, finiteness dimension, local cohomology, attached primes, Cohen-Macaulay ring, Gorenstein ring}
\thanks{The second author was partially supported by the Vietnam National Foundation 
	for Science and Technology Development (NAFOSTED) under grant number 101.04-2023.08.}
	\begin{abstract}
		A foundational result by C. Huneke and V. Trivedi provides a formula for the depth of an ideal in terms of height, computed over a finite set of prime ideals, for rings that are homomorphic images of regular rings. Building on a result by the first author for local quotients of Cohen-Macaulay rings, this paper first gives a new proof and derives a similar formula for the finiteness dimension. Our main result then establishes the depth formula for non-local rings that are homomorphic images of a finite-dimensional Gorenstein ring.
	\end{abstract}
	
	\maketitle
	
	\section{Introduction}
		Throughout this paper, $R$ will denote a  commutative Noetherian ring, $I$ an ideal of $R$, and $M$ a finitely generated $R$-module. We denote by $\depth_R(I, M)$ the depth of $I$ on $M$, which is the length of a maximal $M$-regular sequence in $I$. It is a fundamental invariant in commutative algebra, connected to local cohomology by the well-known formula
	$$\depth_R(I, M)=\min\{i \ge 0 \mid H_I^i(M)\neq 0\}.$$
	In \cite{HT}, Huneke and Trivedi  established a remarkable connection between depth and height. They showed that if $R$ is a homomorphic image of a regular ring, there exists a finite set of prime ideals $\Lambda_M \subseteq \mathrm{Spec}(R)$, depending only on $M$, such that for any ideal $I$ of $R$,
	$$\mathrm{depth}_R(I, M) = \min_{\p \in \Lambda_M} \{\Ht(\frac{I+\p}{\p}) + \mathrm{depth}M_{\p}\}.$$
	While this formula is powerful, the set $\Lambda_M$ was defined in terms of the generic points of certain Zariski-closed sets, which, while effective, is not explicitly described through more standard invariants of the module $M$.
	
	In a subsequent work \cite{An}, the first author shown that this formula holds for a local ring $(R, \m)$ provided that the so-called pseudo-supports of $M$ \cite{BS1}, $$\Psupp^i_R(M)=\{\p\in\Spec (R)\mid  H^{i-\dim (R/\p)}_{\p R_{\p}}(M_{\p})\neq 0\},$$  are closed subsets of $\Spec(R)$.  
	This condition is satisfied, when $R$ is a quotient of a Cohen-Macaulay local ring, in which case $\Lambda_M$ can be described as $\bigcup_{i=0}^{\dim M} \min \Var(\Ann_R(H_{\m}^i(M)))$.

	The main goal of this paper is to provide new proofs and extensions of this depth formula for rings that are homomorphic images  Gorenstein rings. To achieve this, we first focus on the local case. We provide a new, self-contained proof of the depth formula for a local ring $(R, \m)$ that is a homomorphic image of a Cohen-Macaulay local ring. Our approach, which relies on properties of attached primes of local cohomology modules. As an immediate application of this local result, we derive a formula for the finiteness dimension $f_I(M)$ in terms of this explicit finite set (Corollary \ref{C:f-dim}).
	
Our main result then extends the depth formula to the global setting for rings that are homomorphic images of Gorenstein rings. The set $\Lambda_M$ is describled using associated primes of $\mathrm{Ext}$ modules which is well-defined, i.e., it is independent of the choice of the ambient Gorenstein ring.
	
	\begin{theorem}\label{T:Global_Gorenstein}
		Let $R$ be a Noetherian ring which is a homomorphic image of a Gorenstein ring $S$ of finite dimension. Let $M$ be a finitely generated $R$-module. If we set
		$$\Lambda_M = \bigcup_{i \ge 0} \min \Ass_R(\Ext^i_S(M, S)),$$
		then the set $\Lambda_M$ is finite and independent of the choice of $S$. Furthermore, for any ideal $I$ of $R$,
		$$\mathrm{depth}_R(I, M) = \min_{\p \in \Lambda_M} \{\Ht(\frac{I+ \p}{\p}) + \mathrm{depth}M_{\p}\}.$$
	\end{theorem}

	The paper is structured as follows. Section 2 collects preliminary results on attached primes and local duality. In Section 3, firstly we present our new proof of the depth formula in the local Cohen-Macaulay case and derive the formula for the finiteness dimension. The last part of Section 3 is dedicated to the proof of Theorem \ref{T:Global_Gorenstein}; we first establish that the set $\Lambda_M$ is well-defined and then prove the depth formula in the global setting. Finally, we discuss an open question that arises from our work.
	\section{Preliminaries}
	In this section, we collect essential results on attached primes and local duality. Our primary reference is Brodmann and Sharp \cite{BS}.

Following I. G. Macdonald \cite{Mac}, every Artinian $R$-module $A$ has a minimal secondary representation $A=A_1+\ldots +A_n,$ where $A_i$ is $\p_i$-secondary. The set of prime ideals $\{\p_1,\ldots ,\p_n\}$ is an invariant of $A$, called the set of attached prime ideals of $A$, and is denoted by $\Att_R A$.

\begin{lemma}\label{L:1} \cite[Ch. 7]{BS}
	Let $A$ be an Artinian $R$-module. Then
	\begin{enumerate}[(i)] \rm
		\item \textit{$A\neq 0$ if and only if $\Att_RA \neq \emptyset.$}
		\item \textit{$\min \Att_RA = \min \Var(\Ann_RA).$}
		\item \textit{If $(R, \m)$ is local, then $0 < \ell_R(A) < \infty$ if and only if $\Att_RA = \{\m\}.$}
	\end{enumerate}
\end{lemma}

When $(R, \m)$ is a local ring, the local cohomology modules $H^i_{\m}(M)$ are Artinian for any finitely generated $R$-module $M$ \cite[Thm. 7.1.3]{BS}. This allows us to study their sets of attached primes. If $\R$ is the $\m$-adic completion of $R$, any Artinian $R$-module $A$ has a natural structure as an Artinian $\R$-module.

\begin{lemma} \cite[8.2.5]{BS} \label{L:2}
	Let $(R, \m)$ be a local ring and $A$ an Artinian $R$-module. Then $$\Att_RA=\{\p \cap R \mid \p \in \Att_{\R} A\}.$$
\end{lemma}
	
	The following result from \cite{NQ} is a key ingredient in the proof of our first main theorem.
	
	\begin{lemma}  \label{L:NQ}  We assume that $(R, \frak m)$ is an image of a Cohen-Macaulay local ring. The following statements are equivalent:
	\begin{enumerate}[{(i)}]\rm
		\item {\it $R$ is an image of a Cohen-Macaulay local ring;}
		\item {\it $ \mathrm{Att}_{R_{\frak p}}\big(H^{i-\dim (R/\frak p)}_{\frak p R_{\frak p}}(M_{\frak p})\big)=\big\{\frak q R_{\frak p}\mid \frak q\in\mathrm{Att}_R(H^i_{\frak m}(M)), \frak q \subseteq \frak p\big\}$
			for every finitely generated $R$-module $M$, integer $i\ge 0$ and prime ideal $\frak p$ of $R$;}
		\item {\it $\displaystyle\mathrm{Att}_{\widehat{R}}(H^i_{\frak m}(M))=\bigcup_{\frak p\in\mathrm{Att}_R(H^i_{\frak m}(M))}\mathrm{Ass}_{\widehat{R}}(\widehat{R}/\frak p\widehat{R})$ for every finitely generated $R$-module $M$ and  integer $i\geq 0$.}
	\end{enumerate}
\end{lemma}
	
	The next lemma connects attached primes of local cohomology to associated primes of Ext modules. It follows from Local Duality Theorem (see \cite[11.2.6]{BS}).
	
	\begin{lemma}\label{L:key_duality}
		Let $(R, \m)$ be a local ring that is a homomorphic image of an $n$-dimensional Gorenstein local ring $(S, \n)$. Let $M$ be a finitely generated $R$-module. Then for any $i \ge 0$,
		$$\Att_R(H_{\m}^i(M))=\Ass_R(\Ext_S^{n-i}(M, S)).$$
	\end{lemma}
	
	\section{Proof of the main result}
	In this section, we prove Theorem \ref{T:Local_CM} and its corollary.
	
	\begin{lemma}\label{L:depth_inequality}
		For any prime ideal $\p \in \Supp(M)$, we have 
		$$\mathrm{depth}_R(I, M) \le \Ht(\frac{I+ \p}{\p}) + \mathrm{depth}M_{\p}.$$
	\end{lemma}
	\begin{proof}
		Let $s = \mathrm{depth}M_{\p}$ and $t = \Ht(\frac{I+\p}{\p})$. This means $H^s_{\p R_{\p}}(M_{\p}) \neq 0$, so $\Att_{R_{\p}}(H^s_{\p R_{\p}}(M_{\p}))\neq \emptyset$ by Lemma \ref{L:1} (i). Let $\q \in \min V(I+ \p)$ be such that $\Ht(\frac{\q}{\p}) = t$. By applying Lemma \ref{L:NQ}, we see that $\q R_{\q} \in \Att_{R_{\q}}(H^{s+t}_{\q R_{\q}}(M_{\q}))$, which implies $H^{s+t}_{\q R_{\q}}(M_{\q}) \neq 0$. Therefore, $\depth M_{\q} \le s+t$. Since $\depth_R(\q, M) \le \depth M_{\q}$, and $I \subseteq \q$, we have $\depth_R(I, M) \le \depth_R(\q, M) \le s+t$.
	\end{proof}
The depth formula of Huneke and Trivedi was studied by the first author in \cite{An} for any local ring whose pseudo-supports are closed. In that work, the set $\Lambda_M$ was described in terms of these pseudo-supports. In particular, when $(R, \m)$ is a local ring which is a homomorphic image of a Cohen-Macaulay ring, the Huneke-Trivedi depth formula is known to hold with $\Lambda_M$ being the set of minimal attached primes of local cohomology modules. Using the shifted localization principle for attached primes of local cohomology modules, introduced by L. T. Nhan and P. H. Quy (see Lemma \ref{L:NQ}), we provide a new proof for this result. This approach also allows us to derive an application for studying the finiteness dimension formula.
	\begin{proposition}\label{T:Local_CM}
	Let $(R,\m)$ be a Noetherian local ring which is a homomorphic image of a Cohen-Macaulay local ring. Let $M$ be a finitely generated $R$-module. If we set
		$\Lambda_M = \bigcup_{i \ge 0} \min \Att_R(H^i_{\m}(M))$, \linebreak	then for any ideal $I$ of $R$,
	$$\mathrm{depth}_R(I, M) = \min_{\p \in \Lambda_M} \{\Ht(\frac{I+ \p}{\p}) + \mathrm{depth}M_{\p}\}.$$
\end{proposition}

	\begin{proof}
		By Lemma \ref{L:depth_inequality}, we only need to prove the inequality
		$$\mathrm{depth}_R(I, M) \geq \min_{\p \in \Lambda_M} \{\Ht(\frac{I+\p}{\p}) + \mathrm{depth}M_{\p}\}.$$
		Let $t = \mathrm{depth}_R(I, M)$. This means there exists an $M$-regular sequence $x_1, \ldots, x_t$ in $I$. Thus, $I$ consists of zero-divisors on the module $M/(x_1, \ldots,x_t)M$. This implies there exists a prime ideal $\q \in \Ass_R(M/(x_1, \ldots,x_t)M)$ such that $I \subseteq \q$. By standard properties of depth, $\depth M_{\q} = t$, which implies $H^t_{\q R_{\q}}(M_{\q}) \neq 0$.
		
		By Lemma \ref{L:1}(i), there is a prime $\p_0 R_{\q} \in \Att_{R_{\q}}(H^t_{\q R_{\q}}(M_{\q}))$. Let $\p \in \Spec(R)$ be such that $\p R_{\q}$ is a minimal element of $\Att_{R_{\q}}(H^t_{\q R_{\q}}(M_{\q}))$. Using Lemma \ref{L:NQ}, we deduce that $\p \in \min \Att_R(H^{t+\dim R/\q}_{\m}(M))$. Therefore, $\p \in \Lambda_M$.
		
		Applying Lemma \ref{L:NQ} again, we find that $H^{t - \dim(R/\p) + \dim(R/\q)}_{\p R_{\p}}(M_{\p}) \neq 0$, which is equivalent to $H^{t - \Ht(\q/\p)}_{\p R_{\p}}(M_{\p}) \neq 0$. This gives the inequality $\depth M_{\p} \le t - \Ht(\q/\p)$. Rearranging, we get:
		$$t \ge \Ht(\q/\p) + \depth M_{\p}.$$
		Since $I \subseteq \q$, we have $\Ht(\frac{I+\p}{\p}) \le \Ht(\frac{\q+\p}{\p}) = \Ht(\q/\p)$. Thus,
		$$t \ge \Ht(\frac{I+\p}{\p}) + \depth M_{\p}.$$
		As this holds for a specific $\p \in \Lambda_M$, it certainly holds for the minimum over all primes in $\Lambda_M$. The proof is complete.
	\end{proof}

The finiteness dimension $f_I(M) = \inf\{i \mid H_I^i(M) \text{ is not finitely generated}\}$ of a finitely generated $R$-module $M$ with respect to an ideal $I$ is a fundamental invariant that measures the boundary between finite and non-finite local cohomology.
The foundational result in this area is Grothendieck's Finiteness Theorem (see \cite[9.5.2]{BS}). It asserts that for a ring $R$ which is a homomorphic image of a regular ring, the finiteness dimension can be calculated by a depth formula. Specifically, $f_I(M)$ is the minimum of the values $\depth(M_\p) + \Ht((I+\p)/\p)$ taken over all prime ideals $\p \in \Supp(M) \setminus \Var(I)$.

 The corollary below establishes the finiteness dimension $f_I(M)$ when $R$ is a homomorphic image of a Cohen-Macaulay local ring, using the specific finite set of primes $\Lambda_M$.

\begin{corollary}\label{C:f-dim}
	Let $(R,\m)$ be a local ring which is a homomorphic image of a Cohen-Macaulay local ring and $M$ a finitely generated $R$-module. With $\Lambda_M$ as in Theorem \ref{T:Local_CM}, we have
	$$f_I(M)=\min_{\p\in\Lambda_M \setminus V(I)}\{\Ht(\frac{I+\p}{\p})+\depth(M_{\p})\}.$$
\end{corollary}
\begin{proof}
	Assume that $(R,\m)$ be a local ring which is a homomorphic image of a Cohen-Macaulay local ring. By \cite[Corollary 1.2]{Kaw}, $R$ is universally catenary and has Cohen-Macaulay formal fibers. Set $g(\p) := \Ht(\frac{I+\p}{\p})+\depth(M_{\p})$.
	By \cite[Exercise 9.6.6]{BS}, $f_I(M) = \min_{\p \in \Supp(M) \setminus \Var(I)} \{g(\p)\}.$
	
	Set $s_\Lambda := \min_{\p \in \Lambda_M \setminus \Var(I)} \{g(\p)\}$. We need to prove $f_I(M) = s_\Lambda$.
	
	It is obvious that  $f_I(M) \le s_\Lambda$.
	
	Let $\q \in \Supp(M) \setminus \Var(I)$. Set $R' := R_\q$ and $M' := M_\q$. By Proposition \ref{T:Local_CM}, $$ \mathrm{depth}_{R'}(I', M') = \min_{\p' \in \Lambda_{M'}} \{\Ht_{R'}(\frac{I'+ \p'}{\p'}) + \mathrm{depth}(M'_{\p'})\}  $$
	for any ideal $I'$ of $R'$.
	Note that $$\Lambda_{M'} = \Lambda_{M_\q}=\{\p R_\q \mid \p \in \Lambda_M \text{ and } \p \subseteq \q\}.$$ 
	Let $I' := \p R_\q$. For any $\p'=\p R_{\q} \in \Lambda_{M_\q}$, we have $\mathrm{depth}(M'_{\p'}) = \mathrm{depth}((M_\q)_{\p R_\q}) = \mathrm{depth}(M_\p),$ $ \mathrm{depth}_{R'}(I', M') = \mathrm{depth}_{R_\q}(\q R_\q, M_\q) = \depth(M_\q) $,  and $\Ht_{R'}(\frac{I'+ \p'}{\p'}) = \Ht_{R_\q}(\frac{\q R_\q}{\p R_\q}) = \Ht_R(\q/\p)$. Hence
	$$ \depth(M_\q) = \min_{\p \in \Lambda_M, \p \subseteq \q} \{\Ht(\q/\p) + \depth(M_\p)\}. $$
This implies that
	$$ g(\q) = \depth(M_\q) + \Ht\left(\frac{I+\q}{\q}\right) = \left( \min_{\p \in \Lambda_M, \p \subseteq \q} \{\depth(M_\p) + \Ht(\q/\p)\} \right) + \Ht\left(\frac{I+\q}{\q}\right). $$ 
	Let $\p_1 \in \Lambda_M$ be a prime such that $\p_1 \subseteq \q$ and $\depth(M_\q) = \depth(M_{\p_1}) + \Ht(\q/\p_1)$. Such a $\p_1$ must exist.
	Then we have:
	$$ g(\q) = \depth(M_{\p_1}) + \Ht(\q/\p_1) + \Ht\left(\frac{I+\q}{\q}\right). $$
	Now, we use a standard height inequality that holds in catenary rings (which $R$ is, as a homomorphic image of a Cohen-Macaulay ring):
	$$ \Ht(\q/\p_1) + \Ht\left(\frac{I+\q}{\q}\right) \ge \Ht\left(\frac{I+\p_1}{\p_1}\right). $$
		Applying this inequality, we get:
	$$ g(\q) = \depth(M_{\p_1}) + \left( \Ht(\q/\p_1) + \Ht\left(\frac{I+\q}{\q}\right) \right) \ge \depth(M_{\p_1}) + \Ht\left(\frac{I+\p_1}{\p_1}\right) = g(\p_1). $$
	So, we have shown that for any $\q \in \Supp(M) \setminus \Var(I)$, there exists a $\p_1 \in \Lambda_M$ such that $g(\q) \ge g(\p_1)$. It is obvious that $\p_1 \in \Lambda_M \setminus \Var(I)$.
		This proves that $$ \min_{\q \in \Supp(M) \setminus \Var(I)} \{g(\q)\} \ge \min_{\p \in \Lambda_M \setminus \Var(I)} \{g(\p)\} $$
		or $f_I(M) \ge s_\Lambda$.
		
		Therefore $f_I(M) = s_\Lambda$.
\end{proof}

\begin{remark} Assume that $(R,\m)$ be a local ring which is a homomorphic image of a Cohen-Macaulay local ring. By Proposition \ref{T:Local_CM}  and Corollary \ref{C:f-dim}, we have
$f_I(M) \ge \depth(I, M)$. We will consider when equality holds. We will consider when equality holds.

Assume that there exists a prime ideal $\p_0 \in \Lambda_M \setminus \Var(I)$ such that $\depth(I, M) = g(\p_0)$, where $g(\p) := \Ht(\frac{I+\p}{\p})+\depth(M_{\p})$. This implies:
$$f_I(M) = \min_{\p \in \Lambda_M \setminus \Var(I)} \{g(\p)\} \le g(\p_0)=\depth(I, M) .$$
Hence $f_I(M) = \depth(I, M) = g(\p_0).$ 

Conversely, if the minimum for $\depth(I, M)$ is attained only at primes inside $\Var(I)$.
In this case, for every prime ideal $\p' \in \Lambda_M \setminus \Var(I)$, we have $g(\p') > \depth(I, M)$.
In this case $f_I(M)> \depth(I, M)$.
\end{remark}

Now we go to the global case. Firstly, we prove that the set $\Lambda_M$ in Theorem \ref{T:Global_Gorenstein} is well-defined.
	
	\begin{proposition}\label{P:Lambda_independent}
		Let $R$ be a Noetherian ring which is a homomorphic image of a Gorenstein ring $S$ via a surjection $f: S \to R$. Let $M$ be a finitely generated $R$-module. Then the set
		$$ \Lambda_M = \bigcup_{i \ge 0} \min\, \Ass_R(\Ext^i_S(M, S)) $$
		 is independent of the choice of $S$ and the surjection $f$.
	\end{proposition}
	
	\begin{proof}
	For a prime ideal $\p \in \Spec(R)$, let $H^i_{\p R_\p}(M_\p)$ be the $i$-th local cohomology module of $M_\p$ with respect to the ideal $\p R_\p \subseteq R_\p$. We define the set $\Lambda'_M$ as
		$$ \Lambda'_M := \{ \p \in \Spec(R) \mid \p R_\p \in \bigcup_{i \ge 0} \min\, \Att_{R_\p}(H^i_{\p R_\p}(M_\p)) \}. $$
		This definition is independent of $S$. We will now prove $\Lambda_M = \Lambda'_M$.

		Let $\p \in \Lambda_M$. By definition, there exists an integer $k \ge 0$ such that $\p \in \min\, \Ass_R(\Ext^k_S(M, S))$.
		Let $\q = f^{-1}(\p)$ be the corresponding prime ideal in $S$. The ring $S_\q$ is a Gorenstein local ring, and we have an induced surjection $f_\p: S_\q \to R_\p$.
		
		Since $\p$ is a minimal element of $\Ass_R(\Ext^k_S(M, S))$, localizing at $\p$ implies that $\p R_\p$ is a minimal element of $\Ass_{R_\p}((\Ext^k_S(M, S))_\p)$. The $\Ext$ functor commutes with flat localization, so we have
		$$ (\Ext^k_S(M, S))_\p \cong \Ext^k_{S_\q}(M_\p, S_\q). $$
		Let $d = \dim(S_\q)$. By the Local Duality Theorem applied to the Gorenstein local ring $S_\q$ and the module $M_\p$, we have an equality of prime ideal sets:
		$$ \Ass_{R_\p}(\Ext^k_{S_\q}(M_\p, S_\q)) = \Att_{R_\p}(H^{d-k}_{\p R_\p}(M_\p)). $$
		(The primes are in $R_\p$ because the $\Ext$ module is an $R_\p$-module.)
		Combining these facts, we see that $\p R_\p$ is a minimal element of $\Att_{R_\p}(H^{d-k}_{\p R_\p}(M_\p))$. By the definition of $\Lambda'_M$, this implies that $\p \in \Lambda'_M$. Thus, $\Lambda_M \subseteq \Lambda'_M$.
		
	Let $\p \in \Lambda'_M$. By definition, there exists an integer $i \ge 0$ such that $\p R_\p \in \min\, \Att_{R_\p}(H^i_{\p R_\p}(M_\p))$.
		Let $\q = f^{-1}(\p)$ and $d = \dim(S_\q)$. Let $k = d-i$. By local duality,
		$$ \Att_{R_\p}(H^i_{\p R_\p}(M_\p)) = \Ass_{R_\p}(\Ext^k_{S_\q}(M_\p, S_\q)) = \Ass_{R_\p}((\Ext^k_S(M,S))_\p). $$
		So, $\p R_\p$ is a minimal element of the set $\Ass_{R_\p}((\Ext^k_S(M,S))_\p)$. We know that
		$$ \Ass_{R_\p}((\Ext^k_S(M,S))_\p) = \{ \mathfrak{r} R_\p \mid \mathfrak{r} \in \Ass_R(\Ext^k_S(M,S)), \mathfrak{r} \subseteq \p \}. $$
		We claim that the minimality of $\p R_\p$ in this local set implies the minimality of $\p$ in the global set $\Ass_R(\Ext^k_S(M,S))$. Assume for contradiction that $\p$ is not minimal, i.e., there exists $\mathfrak{r} \in \Ass_R(\Ext^k_S(M,S))$ with $\mathfrak{r} \subsetneq \p$. Then $\mathfrak{r} R_\p$ would be an element of $\Ass_{R_\p}((\Ext^k_S(M,S))_\p)$ and $\mathfrak{r} R_\p \subsetneq \p R_\p$. This contradicts the minimality of $\p R_\p$.
		Therefore, $\p$ must be a minimal prime in $\Ass_R(\Ext^k_S(M,S))$. By definition, this means $\p \in \Lambda_M$. Thus, $\Lambda'_M \subseteq \Lambda_M$.
		
		 we conclude that $\Lambda_M = \Lambda'_M$. Since $\Lambda'_M$ is independent of the choice of $S$ and $f$, so is $\Lambda_M$.
	\end{proof}
		
	\begin{proof}[Proof of Theorem \ref{T:Global_Gorenstein}]
		The independence of $\Lambda_M$ was just shown in Proposition \ref{P:Lambda_independent}. Since $S$ has finite dimension, $\Lambda_M$ is finite set.
		
		 We now prove the depth formula.
		It is a standard result that depth can be computed locally:
		$$ \depth_R(I, M) = \min_{\m \in V(I)} \depth_{R_{\m}}(IR_{\m}, M_{\m}). $$
		For any maximal ideal $\m \in V(I)$, let $\n = f^{-1}(\m)$. The ring $R_{\m}$ is a local homomorphic image of the Gorenstein local ring $S_{\n}$. Let $n_{\m} = \dim S_{\n}$. We can apply the local version of the depth formula, which follows from Lemma \ref{L:key_duality} and Theorem \ref{T:Local_CM}. The corresponding set for the local ring $R_{\m}$ is:
		$$ \Lambda_{M_{\m}} = \bigcup_{i \ge 0} \min \Att_{R_{\m}}(H^i_{\m R_{\m}}(M_{\m})) = \bigcup_{i \ge 0} \min \Ass_{R_{\m}}(\Ext_{S_{\n}}^{n_{\m} - i}(M_{\m}, S_{\n})). $$
		As shown in the proof of Proposition \ref{P:Lambda_independent}, this set is precisely the localization of our global set:
		$$ \Lambda_{M_{\m}} = \{\p R_{\m} \mid \p \in \Lambda_M, \p \subseteq \m \}.$$
		Applying the depth formula at each localization $R_{\m}$ as in Proposition \ref{P:Lambda_independent}
		\begin{align*}
			\depth_R(I, M) &= \min_{\m \in V(I)} \left( \depth_{R_{\m}}(IR_{\m}, M_{\m}) \right) \\
			&= \min_{\m \in V(I)} \left( \min_{\p' \in \Lambda_{M_{\m}}} \{\Ht(\frac{I R_{\m} + \p'}{\p'}) + \depth (M_{\m})_{\p'}\} \right) \\
			&= \min_{\m \in V(I)} \left( \min_{\substack{\p \in \Lambda_M \\ \p \subseteq \m}} \{\Ht(\frac{I R_{\m} + \p R_{\m}}{\p R_{\m}}) + \depth (M_{\m})_{\p R_{\m}}\} \right).
		\end{align*}
		Since $\Ht((\dots)_{\m}) = \Ht(\dots)$ for primes contained in $\m$, and $\depth (M_{\m})_{\p R_{\m}} = \depth M_{\p}$, we can simplify this to:
		$$ \depth_R(I, M) = \min_{\m \in V(I)} \left( \min_{\substack{\p \in \Lambda_M \\ \p \subseteq \m}} \{\Ht(\frac{I+\p}{\p}) + \depth M_{\p}\} \right). $$
		This is a minimum taken over pairs $(\m, \p)$ where $\p \in \Lambda_M$, $\m \in V(I)$, and $\p \subseteq \m$. This is equivalent to taking the minimum over all $\p \in \Lambda_M$ and choosing a suitable $\m \in V(I+\p)$ that contains it. Thus,
		$$ \depth_R(I, M) = \min_{\p \in \Lambda_M} \{\Ht(\frac{I+\p}{\p}) + \depth M_{\p}\}. $$
		This completes the proof.
	\end{proof}
	
We have established the depth formula for local rings that are homomorphic images of Cohen-Macaulay rings and for global rings that are homomorphic images of Gorenstein rings. This naturally leads to the question of whether the result can be extended to the most general setting: non-local rings that are quotients of Cohen-Macaulay rings.

\begin{question}
	Let $R$ be a Noetherian ring which is a homomorphic image of a Cohen-Macaulay ring $S$. Let $M$ be a finitely generated $R$-module, and let
	$$\Lambda_M = \bigcup_{i \ge 0} \min \Ass_R(\Ext^i_S(M, S)).$$
	Is it true that the set $\Lambda_M$ is finite, independent of the choice of $S$, and for any ideal $I$ of $R$, the following formula holds?
	$$\mathrm{depth}_R(I, M) = \min_{\p \in \Lambda_M} \{\Ht(\frac{I+ \p}{\p}) + \mathrm{depth}M_{\p}\}.$$
\end{question}

For the above question to be well-posed, one must first establish the finiteness and independence of the set $\Lambda_M$. This task is highly non-trivial and remains an open problem for us, even in the local setting.

\end{document}